\documentclass[12pt, reqno, twoside, letterpaper]{amsart}

\usepackage{paperstyle}
\usepackage{graphicx}

\newcommand{\be}{\begin{equation}}
\newcommand{\ee}{\end{equation}}
\newcommand{\dalign}[1]{\[\begin{aligned} #1 \end{aligned}\]}

\newcommand{\fq}{\ensuremath{q_{\,\sharp}}}

\usepackage{mathtools}
\usepackage{todonotes}
\usepackage[norefs,nocites]{refcheck}

\newcommand{\soundconj}[1]{{\tt SC}\llbracket\hskip1pt#1\hskip1pt\rrbracket}

\newcommand\Eqxy{W}
\newcommand\qfl{q_{\,\flat}}
\newcommand\cqfl{\check q_{\,\flat}}

\title[On a conjecture of Soundararajan]
{On a conjecture of Soundararajan}
      
\author[W.~D.~Banks]{William Banks}

\address{Department of Mathematics, 
         University of Missouri, 
         Columbia MO 65211 USA.}

\email{bankswd@missouri.edu}

\author[I.~E.~Shparlinski]{Igor Shparlinski}

\address{Department of Pure Mathematics,
         University of New South Wales,
         Sydney, NSW 2052 Australia}
		 
\email{igor.shparlinski@unsw.edu.au}

\date{\today}

\begin{document}

\begin{abstract}
Building on recent work of A.~Harper (2012), and using
various results of M.-C.~Chang (2014) and H.~Iwaniec (1974) 
on the zero-free regions of \text{$L$-functions}
$L(s,\chi)$ for characters $\chi$ with a smooth modulus $q$,
we establish a conjecture of K.~Soundararajan (2008)
on the distribution of smooth numbers over reduced residue classes
for such moduli $q$. A crucial ingredient in our argument is
that, for such $q$, there is at most one ``problem character''
for which $L(s,\chi)$ has a smaller zero-free region. Similarly,
using the ``Deuring-Heilbronn'' phenomenon on the repelling nature
of zeros of $L$-functions close to one, we also show that 
Soundararajan's conjecture holds for a family of moduli having 
Siegel zeros.
\end{abstract}


%

\maketitle

\tableofcontents



{\Large\section{Introduction}}

\subsection{Set-up and background} 

Let $P(n)$ denote the largest prime factor of the natural
number $n\ge 2$, and put $P(1)\defeq 1$.
A number $n$ is said to be \emph{$y$-smooth} if $n$ has no prime
factor exceeding~$y$, that is, if $P(n)\le y$.

Let $\cS(y)$ denote the set of all $y$-smooth numbers,
and let $\cS(x,y)$ be the set of $y$-smooth numbers not
exceeding $x$:
$$
\cS(x,y)\defeq\cS(y)\cap[1,x].
$$
As usual, we use $\Psi(x,y)$ to denote the cardinality of $S(x,y)$. 

In this note, we study the distribution of smooth numbers over
arithmetic progressions $a\bmod q$ with the
coprimality condition $(a,q)=1$. Defining
$$
\Psi(x,y;q,a)\defeq\ssum{n\in\cS(x,y)\\n\equiv a\bmod q}1
\mand
\Psi_q(x,y)\defeq\ssum{n\in\cS(x,y)\\(n,q)=1}1,
$$
it is expected that the asymptotic relation
\be\label{eq:prediction}
\Psi(x,y;q,a)\sim\frac{\Psi_q(x,y)}{\varphi(q)} 
\ee
holds (with $\varphi$ the Euler function) under the
condition $(a,q)=1$, over a wide range in the parameters $x,y,q,a$.

Soundararajan\cite[Conjecture~I(A)]{Sound}
has proposed the following conjecture.

\begin{conjecture}\label{conj:one}
{\rm (Soundararajan)} For any fixed value of $A>0$, if
$q$ is sufficiently large $($depending only on $A)$ with
$q\le y^A$ and $(a,q)=1$, then~\eqref{eq:prediction} holds as $\log x/\log q\to\infty$. 
\end{conjecture}

Earlier, Granville~\cite{Gran1, Gran2} had established this result
for $A<1$, and he pointed out that the proof for
arbitrarily large values of $A$ must lie fairly deep, for it implies
Vinogradov's conjecture that the least quadratic
nonresidue  modulo $p$ is of size $p^{o(1)}$. 
Soundararajan~\cite{Sound} has shown that~\eqref{eq:prediction} holds for moduli
$$
q\le y^{4\sqrt{e}-\eps}
$$
provided that
\be\label{eq:ticklish}
y^{(\log\log y)^4}\le x\le \exp(y^{1-\eps}).
\ee
In~\cite{Harper} Harper demonstrates how to remove the 
conditions~\eqref{eq:ticklish}, thereby
settling Soundararajan's Conjecture for all $A<4\sqrt{e}$.

Short of an improvement of the Burgess bound on character sums,
one can consider the following variant of Soundararajan's Conjecture
in which the moduli all belong to a prescribed subset $\cQ$ of
the natural numbers $\N$.

\bigskip\noindent{\sc Hypothesis $\soundconj{\cQ}$.} 
Let $\cQ\subseteq\N$. For every fixed $A>0$, if $q$ is
sufficiently large $($depending only on $A)$ with
$$
q\le y^A,\qquad(a,q)=1,\qquad q\in\cQ,
$$
then~\eqref{eq:prediction} holds as $\log x/\log q\to\infty$.
\bigskip

Note that Soundararajan's Conjecture is nothing but
$\soundconj{\N}$, and it is clear that $\soundconj{\N}$ holds if
and only if $\soundconj{\cQ}$ is true for every
set $\cQ\subseteq\N$. We say that \emph{Soundararajan's Conjecture
holds over $\cQ$} whenever $\soundconj{\cQ}$ is true.

\subsection{New results} 

We start by establishing $\soundconj{\cQ}$ for a class
of sufficiently smooth moduli.

\begin{theorem}\label{thm:chang-style}
For every fixed value of $A>0$, there is a number $Q_A>0$
$($depending only on $A)$ for which the following holds.
If
\be\label{eq:ch-restrict}
P(q)^{Q_A}<q\le y^A\mand (a,q)=1,
\ee
then the asymptotic relation~\eqref{eq:prediction} holds as $\log x/\log q\to\infty$.
\end{theorem}

\begin{corollary}\label{cor:chang-style}
Let $\cQ$ be a set of natural numbers $q$ with the property that
$$
\log P(q)=o(\log q)\qquad(q\to\infty,~q\in\cQ).
$$
Then Soundararajan's Conjecture holds over  $\cQ$. 
\end{corollary}

An important special case of Corollary~\ref{cor:chang-style}
is the set $\cQ\defeq p^\N$ consisting of all
powers of a fixed prime $p$.
In fact, our work in the present paper
has been initially motivated by a series of results on
arithmetic problems involving progressions modulo large powers
of a fixed prime. Such results include:
\begin{itemize}
\item bounds on the zero-free regions of
$L$-functions, which leads to results on the distribution of primes in 
arithmetic progressions (see~\cite{BanksShpar1,BanksShpar2,Gal,Iwan}),
\item asymptotic formulas in the Dirichlet problem on sums with the
divisor function over arithmetic progressions modulo $p^n$
(see~\cite{Kh, LSZ}),
\item  asymptotic formulas for moments of $L$-functions
(see~\cite{Mil,MilWhi}).
\end{itemize}

It is clear that, for any given set $\cQ$,
to establish $\soundconj{\cQ}$ one must show that the following
hypothesis holds for all large $A$. 

\bigskip\noindent
{\sc Hypothesis $\soundconj{\cQ,A}$.} 
Fix $\cQ\subseteq\N$ and $A>0$. There is a number
$Q_A>0$ such that if
\be\label{eq:SQA-ass}
Q_A<q\le y^A,\qquad(a,q)=1,\qquad q\in\cQ,
\ee
then~\eqref{eq:prediction} holds as $\log x/\log q\to\infty$.
\bigskip

Since the number of moduli $q\le x$ with $q\ge P(q)^{Q_A}$
is $\Psi(x,x^{1/Q_A})\sim\rho(Q_A)x$, where $\rho$ is the
Dickman function, Theorem~\ref{thm:chang-style} implies that the
following variant of Soundararajan's Conjecture (with $A$ arbitrary but fixed) holds over a set of positive asymptotic density.

\begin{corollary}\label{cor:pos dens}
For any fixed $A>0$, there is a set $\cQ\subseteq\N$
of positive asymptotic density for which
$\soundconj{\cQ,A}$ holds.
\end{corollary}

Corollary~\ref{cor:pos dens} complements certain
Bombieri-Vinogradov type results due to Granville and
Shao~\cite{GranShao} and Harper~\cite{Harper2},
which imply \eqref{eq:prediction} for a set of moduli~$q$ of
asymptotic density one, but in more restrictive ranges of $y$.
For example, none of those results apply to very smooth numbers
with (say) $y$ of size $(\log x)^{o(1)}$.

The next result asserts that the asymptotic relation~\eqref{eq:prediction} holds as $q$ varies over a set
of exceptional moduli.

\begin{theorem}\label{thm:siegel-style}
For every fixed value of $A>0$, there is a number $Q_A>0$
$($depending only on $A)$ for which the following holds.
If
\be\label{eq:restrictions2}
Q_A<q\le y^A\mand (a,q)=1,
\ee
and there is a character $\chi$ modulo $q$ such that
$L(s,\chi)$ has a zero $\beta+i\gamma$ of $L(s,\chi)$
satisfying
\be\label{eq:except-zero-beta}
\beta>1-\frac{Q_A^{-1}}{\log q(|\gamma|+3)},
\ee
then the asymptotic relation~\eqref{eq:prediction} holds as $\log x/\log q\to\infty$.
\end{theorem}

\begin{corollary}\label{cor:siegel-style}
Let $\cQ$ be a set of natural numbers such that,
for every $q\in\cQ$, there is a character $\chi$ modulo $q$
and a real zero $\beta_q$ of $L(s,\chi)$ satisfying
$$
(1-\beta_q)\log q=o(1)\qquad(q\to\infty,~q\in\cQ).
$$
Then Soundararajan's Conjecture holds over  $\cQ$. 
\end{corollary}

In particular, Corollary~\ref{cor:siegel-style}
shows that any future work on 
Soundararajan's Conjecture (over $\N$) can assume that
Siegel zeros do not exist.

\begin{remark}
It also true that Soundararajan's Conjecture is
true if one assumes the Extended Riemann Hypothesis. This is easily
proved using Proposition~\ref{prop:SQA-criterion} in
\S\ref{sec:ests-various-ranges} below.
\end{remark}

Our proofs of Theorems~\ref{thm:chang-style} and
\ref{thm:siegel-style} rely on the argument of Harper~\cite{Harper}
(which in turn builds upon original ideas of 
Soundararajan~\cite{Sound}). The treatment of
the so-called ``problem range'' is the primary issue
(see \S\ref{sec:init-discuss} below), thus
a major part of the proof of Theorem~\ref{thm:chang-style}
is devoted to elimination of this range. This is accomplished
via a combination of results of Chang~\cite{Chang} and 
Iwaniec~\cite{Iwan}, which give bounds on certain character sums
and on the zero-free regions of $L$-functions modulo highly composite
integers. We remark that, for a slightly more restrictive class
of moduli, some stronger bounds have been obtained by
the authors (see~\cite{BanksShpar1,BanksShpar2}), but these do not
lead to better results on Soundararajan's Conjecture.
Concerning Theorem~\ref{thm:siegel-style}, our proof exploits
the ``Deuring-Heilbronn'' phenomenon on the repelling nature of
zeros of $L$-functions close to one.


{\Large\section{Preliminaries}}

\subsection{Notation}
In what follows, given functions $F$ and $G>0$ we use
the equivalent notations $F=O(G)$ and $F\ll G$ to signify that the
inequality $|F|\le c\,G$ holds with some constant $c>0$.
Throughout the paper, any implied constants may depend on the
parameters $A$ and $\Phi$ but are independent of other variables. 

We also write $F\asymp G$ or $F=\Theta(G)$ 
whenever $F,G>0$ and we have both $F=O(G)$ and $G=O(F)$.

The notations $F\sim G$ and $F=o(G)$ are used to indicate that 
$F/G\to 1$ and $F/G\to 0$, respectively, as certain specified
parameters tend to infinity.


\subsection{Initial discussion}
\label{sec:init-discuss}

Harper's theorem~\cite[Theorem~1]{Harper} implies
$\soundconj{\cQ,A}$ for any set $\cQ$ and any $A<4\sqrt{e}$.
Thus, in what follows we can assume that
$A\ge 4\sqrt{e}$ and $y^{4\sqrt{e}}\le q\le y^A$.
In particular, the parameters
$$
u\defeq\frac{\log x}{\log y}
\mand
v\defeq\frac{\log x}{\log q}
$$
are comparable in size (that is, $u\asymp v$)
since $A\ge u/v\ge 4\sqrt{e}$,
and so $u\to\infty$ if and only if $v\to\infty$.

We remark that, in this section and the next, $y$
is sometimes required to exceed a large number
that might depend on $A$. However, in view of~\eqref{eq:SQA-ass},
we can begin by taking $Q_A$ large enough to guarantee
that $y$ meets these requirements.
 
For any character $\chi$ modulo $q$ we put
$$
\Psi(x,y;\chi)\defeq\ssum{n\in\cS(x,y)}\chi(n).
$$
In particular, $\Psi_q(x,y)=\Psi(x,y;\chi_0)$, where
$\chi_0$ is the principal character. Using
Dirichlet orthogonality we see that~\eqref{eq:prediction}
is equivalent to the assertion that
\be\label{eq:translation}
\ssum{\chi\ne\chi_0}\overline\chi(a)\Psi(x,y;\chi)
=o\bigl(\Psi(x,y;\chi_0)\bigr)\qquad
(u\to\infty),
\ee  

As in~\cite{Harper,Sound} it suffices to
establish a smooth variant of~\eqref{eq:translation}. 
More precisely, let $\Phi:[0,\infty)\to[0,1]$
be a function that is
supported on $[0,2]$, equal to one on $[0,\tfrac12]$,
and such that $\Phi\in C^9$, that is, $\Phi$ is
nine times continuously differentiable.
For every character $\chi$ modulo $q$ we denote
$$
\Psi(x,y;\chi,\Phi)
\defeq\ssum{n\in\cS(y)}\chi(n)\Phi(n/x).
$$
Then, it is enough show that
\be\label{eq:trans-with-Phi}
\ssum{\chi\ne\chi_0}\overline\chi(a)\Psi(x,y;\chi,\Phi)
=o\bigl(\Psi(x,y;\chi_0,\Phi)\bigr)\qquad
(u\to\infty)
\ee
holds, since the passage from~\eqref{eq:trans-with-Phi} back to~\eqref{eq:translation} can be accomplished using the unsmoothing
method outlined by Harper~\cite[Appendix~A]{Harper}.

As in~\cite{Harper,Sound} we start by writing
\be\label{eq:owl}
\Psi(x,y;\chi,\Phi)
=\frac{1}{2\pi i}\int_{c-i\infty}^{c+i\infty}
L(s,\chi;y)x^s\breve\Phi(s)\,ds\qquad(c>0)
\ee
for any $c>0$, where
$$
L(s,\chi;y)\defeq\prod_{p\le y}(1-\chi(p)p^{-s})^{-1}
=\sum_{n\in\cS(y)}\chi(n)n^{-s}
$$
and
$$
\breve\Phi(s)\defeq\int_0^\infty\Phi(t)t^{s-1}\,dt.
$$
Note that the bound
\be\label{eq:brevePhibound}
\breve\Phi(s)\ll |s|^{-1}(|s|+1)^{-8}
\ee
follows from our smoothness
assumption on $\Phi$ (using integration by parts and the continuity
of $\Phi^{(9)}$). We choose $c$ to be $\alpha\defeq\alpha(x,y)$,
the unique positive solution to the equation
$$
\sum_{p\le y}\frac{\log p}{p^\alpha-1}=\log x.
$$ 
The quantity $\alpha$ is introduced in a saddle point argument of
Hildebrand and Tenenbaum~\cite{HilTen1} for $\Psi(x,y)$
(see also~\cite{HilTen2}), and it has been
applied by de la Bret\'eche and Tenenbaum~\cite{delaBTenen}
to $\Psi_q(x,y)$ for an arbitrary modulus $q$.
Using only the trivial bound $\omega(q)\ll\log q$
on the number of distinct prime factors of~$q$,
from~\eqref{eq:SQA-ass} it follows immediately that
$\omega(q)\ll y^{1/2}/\log y$ (the implied constant is
independent of $A$ if (say) $Q_A\ge A^{4A}$). Therefore,
$q$ satisfies one of the conditions $(C_1)$ or $(C_2)$ of~\cite[Corollaire~2.2]{delaBTenen}, 
and so an application of~\cite[Th\'eor\`eme~2.1]{delaBTenen} allows us to conclude that

\be\label{eq:psychobabble}
\Psi(x,y;\chi_0)\asymp\Psi(x,y)\prod_{p\,\mid\,q}(1-p^{-\alpha})
\ee
provided both quantities $y$ and $u$
exceed a certain absolute constant; note that
the implied constants in~\eqref{eq:psychobabble} are absolute.
Combining~\eqref{eq:psychobabble} with~\cite[Theorem~1]{HilTen1}
it follows that 
\be\label{eq:airplants}
\Psi(x,y;\chi_0)
\asymp\Eqxy\defeq\frac{x^\alpha L(\alpha,\chi_0;y)}
{\alpha\sqrt{(1+\log x/y)\log x\log y}}.
\ee
Since $0<\alpha\ll 1$
the quantities $\Psi(\tfrac12x,y;\chi_0)$ and $\Psi(2x,y;\chi_0)$ are 
comparable in size; thus, as $\ind{[0,\frac12]}\le\Phi\le\ind{[0,2]}$ 
one finds that
$$
\Psi(x,y;\chi_0,\Phi)\asymp\Eqxy.
$$

Following~\cite{Harper, Sound} we now denote
$$
\Xi_q(k)\defeq\cX_q(k)\setminus\cX_q(k+1)
\qquad(0\le k\le\tfrac12\log q)
$$
with
$$
\cX_q(k)\defeq\bigl\{\chi\bmod q:
\chi\ne\chi_0,\ L(\sigma+it,\chi)\ne 0\ \text{for}\ 
\sigma>1-k/\log q,\ |t|\le q\bigr\}.
$$
Using~\eqref{eq:owl} (with $c\defeq\alpha$), the sum on the left
side of~\eqref{eq:trans-with-Phi} satisfies the bound
$$
\ssum{\chi\ne\chi_0}\overline\chi(a)\Psi(x,y;\chi,\Phi)\ll 
\sum_{0\le k\le\frac12\log q}
|\Xi_q(k)|\max\limits_{\chi\in\Xi_q(k)}
\biggl|\int_{\alpha-i\infty}^{\alpha+i\infty}
L(s,\chi;y)x^s\breve\Phi(s)\,ds\biggr|.
$$

For the specific moduli considered in
Theorem~\ref{thm:chang-style} and in
Theorem~\ref{thm:siegel-style}, we show
that the $L$-functions $L(s,\chi)$ attached to
characters $\chi$ modulo $q$  have \emph{no zeros close to one},
with at most one exceptional ``problem character'' $\chi_\bullet$
(in the sense of~\cite{Sound}). More precisely, we need to know
(see Proposition~\ref{prop:SQA-criterion}
 in \S\ref{sec:ests-various-ranges} below):
\be\label{eq:k0defined}
\cA\defeq
\bigcup\limits_{k<k_0}\Xi_q(k)=\varnothing\text{~or~}\{\chi_\bullet\}
\qquad\text{with}\quad
k_0\defeq \rf{4A\log A+D},
\ee
where $D$ is the absolute constant described
in~\cite[Rodosski\u{\i} Bound~1]{Harper}.
This means that the ``problem range'' (in the sense of 
Harper~\cite{Harper}) can be handled easily
(this is definitely \emph{not} the
case in situations where~\eqref{eq:k0defined} fails).
It is precisely for this reason that we have been able to
remove the obstruction $A<4\sqrt{e}$ encountered in 
the previous papers~\cite{Harper, Sound}
in the case that $\cQ\defeq\N$.

Assume from now on that $\cQ$ is a set of
natural numbers such that~\eqref{eq:k0defined} holds
for every $q\in\cQ$.
Taking the above considerations into account,
and introducing the notation
$$
\cI(\chi)\defeq\int_{\alpha-i\infty}^{\alpha+i\infty}
L(s,\chi;y)x^s\breve\Phi(s)\,ds\qquad(\chi\ne\chi_0),
$$
to verify~\eqref{eq:trans-with-Phi} (and 
establish $\soundconj{\cQ,A}$) it suffices to show that
\be\label{eq:bitcoin1}
\big|\cI(\chi_\bullet)\big|
=o(\Eqxy)\qquad(u\to\infty)
\ee
holds for a problem character $\chi_\bullet\in\cA$, and that
\be\label{eq:bitcoin2}
\sum_{k\ge k_0}
|\Xi_q(k)|\max\limits_{\chi\in\Xi_q(k)}
\big|\cI(\chi)\big|
=o(\Eqxy)\qquad(u\to\infty),
\ee 
where $\Eqxy$ is defined by~\eqref{eq:airplants}.
For the most part, the results we need are already contained in~\cite{Harper}; these are briefly reviewed in
\S\ref{sec:ests-various-ranges}.

In what follows, we use the terminology (cf.~\cite[p.~184]{Harper}) 
\begin{itemize}
\item $y$ is ``large'' if $e^{(\log x)^{0.1}}<y\le x$;
\item $y$ is ``small'' if
$(\log\log x)^3\le y\le e^{(\log x)^{0.1}}$;
\item $y$ is ``very small'' if $y<(\log\log x)^3$.
\end{itemize}
We also say (cf.~\cite[p.~186]{Harper}) that $k$ lies in
\begin{itemize}
\item the ``basic range'' if $\sqrt{u}\le k\le\tfrac12\log q$;
\item the ``Rodosski\u{\i} range'' if
$k_0\le k<\sqrt{u}$,
\item the ``problem range'' if $0\le k<k_0$,
\end{itemize}
where $k_0\defeq \rf{4A\log A+D}$ as in~\eqref{eq:k0defined}.
Again, we emphasize that our proofs of
Theorems~\ref{thm:chang-style} and~\ref{thm:siegel-style}
essentially amount to showing that the problem range contains
\emph{at most one} character $\chi_\bullet$ in each case.


\subsection{Reduction to characters in the problem range} 
\label{sec:ests-various-ranges}

Both Theorems~\ref{thm:chang-style} and~\ref{thm:siegel-style} follow 
from the following general statement (which does not assume
anything about the arithmetic structure of the modulus $q$).

\begin{proposition}\label{prop:SQA-criterion}
Fix $\cQ\subseteq\N$ and $A>0$. Suppose that, for every
$q\in\cQ$, there is at most one nonprincipal
character $\chi_\bullet$ modulo $q$
for which the $L$-function $L(s,\chi_\bullet)$ has a zero
$\beta+i\gamma$ in the rectangle
$$
\beta>1-\frac{\rf{4A\log A+D}}{\log q},\qquad
|\gamma|\le q.
$$
Then $\soundconj{\cQ,A}$ is true.
\end{proposition}

\begin{proof} 
We outline what is needed to
establish~\eqref{eq:bitcoin1} and~\eqref{eq:bitcoin2}
for $y$ lying in various ranges; the proposition follows.
The underlying ideas are due to Soundararajan~\cite{Sound},
and subsequent refinements are due to Harper~\cite{Harper}.
To begin, when $y$ is \emph{not} very small, we express
$\cI(\chi)$ as a sum 
$$
\cI(\chi) = \cI_-(\chi)+\cI_0(\chi)+\cI_+(\chi),
$$
where the central integral is
$$
\cI_0(\chi)\defeq\int_{\alpha-i(yq)^{1/4}}^{\alpha+i(yq)^{1/4}}
L(s,\chi;y)x^s\check\Phi(s)\,ds,
$$
and the integral tails are given by
$$
\cI_{\pm}(\chi)\defeq\int_{\alpha\pm i(yq)^{1/4}}^{\alpha\pm i\infty}
L(s,\chi;y)x^s\check\Phi(s)\,ds
$$
for either choice of the sign $\pm$.

\bigskip

\noindent{\sc Case 1:} Integral tails with $y$ not very small
and $k$ arbitrary.

\medskip

In view of~\eqref{eq:brevePhibound}, 
one has (cf.~\cite[p.~186]{Harper})
$$
\cI_{\pm}(\chi)\ll\frac{x^\alpha L(\alpha,\chi_0;y)}{y^2q^2}.
$$
Since $\sum_k|\Xi_q(k)|$ is at most the total number of
characters modulo~$q$, that is,
\be\label{eq:sumkXik-bd}
\sum_{k\ge 0}|\Xi_q(k)|\le \varphi(q),
\ee
and we have (cf.~\cite[p.~185]{Harper})
$$
\alpha=\begin{cases}
\Theta\bigl(\frac{y}{\log x\log y}\bigr)
&\quad\hbox{if $y\le\log x$},\\ \\
1-\frac{\log(u\log u)}{\log y}+O(\frac{1}{\log y})
&\quad\hbox{if $y>\log x$},
\end{cases}
$$
it is immediate that
\be\label{eq:tails,notverysmall}
\sum_{k\ge 0}|\Xi_q(k)|
\max\limits_{\chi\in\Xi_q(k)}
\big|\cI_{\pm}(\chi)\big|
\ll  \frac{\Eqxy}{yq},
\ee
where $\Eqxy$ is defined in~\eqref{eq:airplants}.
Since $x\ge u$ and $y\ge(\log\log x)^3$, we have $y\to\infty$
as $u\to\infty$; this implies that the sums in~\eqref{eq:tails,notverysmall} contribute an amount
of size $o(\Eqxy)$ to both~\eqref{eq:bitcoin1} and~\eqref{eq:bitcoin2}.

\bigskip

\noindent{\sc Case 2:} Central integral with $y$ large
and $k$ in the basic range.

\medskip

In addition to~\eqref{eq:sumkXik-bd},
one needs a strong individual bound on $|\Xi_q(k)|$
for smallish values of $k$.
The papers~\cite{Harper,Sound} use
\be\label{eq:logfreedensity}
|\Xi_q(k)|\le C_1e^{C_2k},
\ee
where $C_1,C_2>0$ are certain 
absolute constants, which is a consequence of the
\text{log-free} density estimate for Dirichlet $L$-functions;
see, for example, Iwaniec and Kowalski~\cite[Chapter~18]{IwanKowal}.
In particular, in terms of the same constant $C_2$,
Harper~\cite[p.~189]{Harper} derives the bound
$$
\cI_0(\chi)
\ll \biggl\{\frac{1}{yq^{1.99}}+e^{-(C_2+1)k}\biggr\}\Eqxy
$$
for all sufficiently large $u$. We multiply this bound by $|\Xi_q(k)|$
and then sum over all $k$ in the basic range, taking into account~\eqref{eq:sumkXik-bd} and~\eqref{eq:logfreedensity},
we get that
$$
\sum_{\text{basic~}k}|\Xi_q(k)|\max\limits_{\chi\in\Xi_q(k)}
\big|\cI_0(\chi)\big|
\ll \biggl\{\frac{1}{yq^{0.99}}+e^{-\sqrt{u}}\biggr\}\Eqxy.
$$
As in Case 1, $y\to\infty$ as $u\to\infty$, 
so the sum in this bound contributes an amount
of size $o(\Eqxy)$ to the sum in~\eqref{eq:bitcoin2}.

\bigskip

\noindent{\sc Case 3:} Central integral with $y$ large
and $k$ in the Rodosski\u{\i} range.

\medskip

As an application of his Rodosski\u{\i} Bound~1
(which combines earlier results of Soundararajan~\cite[Lemmas~4.2 and 4.3]{Sound}),
Harper~\cite[p.~189]{Harper} shows that
$$
\cI_0(\chi)
\ll e^{-\Theta(\sqrt{u\log u}\,)}\Eqxy
$$
holds in the present case (one requires that $y/(A+1)^2$
is sufficiently large, which we can assume).
Using~\eqref{eq:logfreedensity} we get that
$$
\sum_{\text{Rodosski\u{\i}~}k}|\Xi_q(k)|
\max\limits_{\chi\in\Xi_q(k)}\big|\cI_0(\chi)\big|
\ll e^{-\Theta(\sqrt{u\log u}\,)}\Eqxy=o(\Eqxy)
\qquad(u\to\infty).
$$

\bigskip

\noindent{\sc Case 4:} Central integral with $y$ large
and $k$ in the problem range.

\medskip

Following Harper~\cite[p.~191]{Harper} we define $\cA$ as
in~\eqref{eq:k0defined} and put $B\defeq |\cA|\le 1$.
When $B=0$, there is nothing to do.
When $B=1$, using~\cite[Rodosski\u{\i} Bound~2]{Harper}
instead of~\cite[Rodosski\u{\i} Bound~1]{Harper}, and
arguing as in~\cite[\S2.4]{Harper}, we derive the individual bound
\be\label{eq:bullet-one}
\cI_0(\chi_\bullet)
\ll e^{-\Theta(\sqrt{u\log u}\,)}\Eqxy,
\ee
which suffices to establish~\eqref{eq:bitcoin1}.

\begin{remark}
It is worth reiterating that our use of~\cite[Rodosski\u{\i} Bound~2]{Harper} to derive~\eqref{eq:bullet-one} relies on the fact that the character
$\chi_\bullet$ (if it exists) has order \emph{two}, which exceeds
the cardinality $B$ of $\cA$.
In other words, the set $\cB$ defined in~\cite[Section~2.5]{Harper} 
is \emph{empty}.
\end{remark}

\bigskip

\noindent{\sc Case 5:} Central integral with $y$ small
and $k$ arbitrary.

\medskip

Building on ideas of Soundararajan~\cite{Sound},
Harper~\cite[\S2.6]{Harper} proves that
\be\label{eq:promises}
\cI_0(\chi)
\ll\Eqxy\begin{cases}
2^{-y^{1/3}}
&\quad\hbox{if $(\log\log x)^3\le y\le \log x$},\\
2^{-(\log y)^4}
&\quad\hbox{if $\log x<y\le e^{(\log x)^{0.1}}$},
\end{cases}
\ee
holds for every $\chi\ne\chi_0$ when $k\ge k_0$.
Again, using~\cite[Rodosski\u{\i} Bound~2]{Harper}
in place of~\cite[Rodosski\u{\i} Bound~1]{Harper},
it is further shown that the same individual bound holds for
$\chi_\bullet$. As in Case 1 we have
$y\to\infty$ as $u\to\infty$, hence~\eqref{eq:bitcoin1} follows.

Using~\eqref{eq:sumkXik-bd} and~\eqref{eq:promises} we also get that
$$
\sum_{k\ge k_0}|\Xi_q(k)|
\max\limits_{\chi\in\Xi_q(k)}
\big|\cI_0(\chi)\big|
\ll\Eqxy\begin{cases}
q\,2^{-y^{1/3}}
&\quad\hbox{if $(\log\log x)^3\le y\le \log x$},\\
q\,2^{-(\log y)^4}
&\quad\hbox{if $\log x<y\le e^{(\log x)^{0.1}}$}.
\end{cases}
$$
Since $y\to\infty$ and $q\le y^A=o\bigl(2^{(\log y)^4}\bigr)$,
the above sum contributes $o(\Eqxy)$ to~\eqref{eq:bitcoin2}.

\bigskip

\noindent{\sc Case 6:} Full integral with $y$ very small
and $k$ arbitrary.

\medskip

Harper~\cite[\S2.6]{Harper} shows that $\cI_0(\chi)\ll\Eqxy/\log x$
holds for $1\ll y\le(\log\log x)^3$, and thus
$$
\sum_{k\ge 0}|\Xi_q(k)|
\max\limits_{\chi\in\Xi_q(k)}
\big|\cI_0(\chi)\big|
\ll \frac{q}{\log x}\,\Eqxy.
$$
Since $q\le (\log\log x)^{3A}$ in this case,
the sum here contributes an amount of size $o(\Eqxy)$
to the sum in~\eqref{eq:bitcoin2}. In view of~\eqref{eq:tails,notverysmall} we complete the
 proof of~\eqref{eq:bitcoin1} and~\eqref{eq:bitcoin2} in this situation
provided that $yq\to\infty$ as $u\to\infty$, e.g., whenever
$q\sqrt{y}>(\log\log x)^{1/3}$.

When $q\sqrt{y}\le(\log\log x)^{1/3}$, we express
$\cI(\chi)$ as a sum $\cJ_\infty(\chi)+\cJ_0(\chi)$ with
\dalign{
\cJ_\infty(\chi)&\defeq\int\limits_{|t|\ge 1}
L(\alpha+it,\chi;y)x^{\alpha+it}\breve\Phi(\alpha+it)\,dt,\\
\cJ_0(\chi)&\defeq\int\limits_{|t|<1}
L(\alpha+it,\chi;y)x^{\alpha+it}\breve\Phi(\alpha+it)\,dt,
}
and we apply the method of Harper in~\cite[\S2.7]{Harper}
with $\eps\defeq 1$ in his notation. Harper shows that
$$
\cJ_\infty(\chi)\ll\frac{\sqrt{y}}{(\log\log x)^{2/5}}\,\Eqxy,
$$
hence by~\eqref{eq:sumkXik-bd} one has
$$
\sum_{k\ge 0}|\Xi_q(k)|
\max\limits_{\chi\in\Xi_q(k)}
\big|\cJ_\infty(\chi)\big|
\ll\frac{q\sqrt{y}}{(\log\log x)^{2/5}}\,\Eqxy
\le\frac{\Eqxy}{(\log\log x)^{1/15}}.
$$
Thus, this sum contributes an amount of size $o(\Eqxy)$ to 
both~\eqref{eq:bitcoin1} and~\eqref{eq:bitcoin2}. On the other hand, 
arguing as in~\cite[\S2.7]{Harper} and applying~\cite[Rodosski\u{\i} Bound~1]{Harper}
or~\cite[Rodosski\u{\i} Bound~2]{Harper} as appropriate,
for some small constant $c\in(0,1)$ we have
$$
\biggl|\frac{L(\alpha+it,\chi;y)}{L(\alpha,\chi_0;y)}\biggr|
\ll\biggl(1+\frac{\log x}{y}\biggr)^{-0.2c\sqrt{y}}.
$$
Therefore, for $100/c^2\le y<(\log\log x)^3$, and
keeping in mind the definition~\eqref{eq:airplants},
we get that
$$
\cJ_0(\chi)
\ll\frac{(\log\log x)^{3/2}}{\log x}\,\Eqxy.
$$
Using~\eqref{eq:sumkXik-bd} again, it follows that
$$
\sum_{k>k_0}|\Xi_q(k)|
\max\limits_{\chi\in\Xi_q(k)}
\big|\cJ_0(\chi)\big|
\ll\frac{q\,(\log\log x)^{3/2}}{\log x}\,\Eqxy.
$$
This sum also contributes an amount of size $o(\Eqxy)$
to both~\eqref{eq:bitcoin1} and~\eqref{eq:bitcoin2},
and we are done.
\end{proof}


{\Large\section{Characters in the problem range}}

\subsection{Exceptional zeros}
We apply two familiar principles that are commonly used in
treatments of Linnik's Theorem. The first is the
zero-free region for Dirichlet $L$-functions
(see Gronwall~\cite{Gron}, Landau~\cite{Landau} and
Titchmarsh~\cite{Titch}).

 \begin{lemma}\label{lem:gron-land-titch}
There is an absolute constant $c_1>0$ such that,
for every $q\in\N$, the function
\be\label{eq:Lproduct}
\prod_{\chi\bmod q}L(s,\chi)
\ee
has at most one zero $\beta+i\gamma$ satisfying
$$
\beta>1-\frac{c_1}{\ell}\qquad\text{with}\quad
\ell\defeq\log q(|\gamma|+3).
$$
Such a zero, if one exists, is simple and real, and 
corresponds to a nonprincipal real character.
\end{lemma}

The second principle, which is due to Linnik~\cite{Linnik},
is often referred to as the ``Deuring-Heilbronn'' phenomenon.

\begin{lemma}\label{lem:linnik}
There is an absolute constant $c_2>0$ for which
the following holds. Suppose the exceptional zero in
Lemma~\ref{lem:gron-land-titch} exists and is $($say$)$
$\beta=1-\eps/\log q$. Then
the function~\eqref{eq:Lproduct} does not vanish
in the region
$$
\sigma>1-\frac{c_2\log(\eps^{-1})}{\ell}\qquad\text{with}\quad
\ell\defeq\log q(|\gamma|+3).
$$
\end{lemma}

\subsection{Bounds on character sums}

The proof of Theorem~\ref{thm:chang-style} relies heavily
on the following result
of Chang~\cite[Corollary~9]{Chang}, which bounds short
character sums over intervals for certain primitive
characters with a smooth conductor.

\begin{lemma}\label{lem:chang-char-bound}
There are absolute constants $c_1,c_2>0$
for which the following holds.
Let $\chi$ be a primitive character modulo $q$, let $T\ge 1$,
and let $\cI$ be an arbitrary interval of length $N$, where
$q>N>P(q)^{1000}$ and
$$
\log N\ge(\log qT)^{1-c_1}+c_2\log\Bigl(\frac{2\log q}{\log\fq}\Bigr)
\frac{\log\fq}{\log\log q}\qquad\text{with}
\quad \fq\defeq\prod_{p\,\mid\,q}p.
$$
Then
\be\label{eq:SCS1}
\biggl|\sum_{n\in\cI}\chi(n)n^{-it}\biggr|
\le N e^{-\sqrt{\log N}}\qquad(|t|\le T).
\ee
\end{lemma}

We apply the following corollary of Lemma~\ref{lem:chang-char-bound},
which provides a weaker bound but has the advantage that it
can be applied to longer intervals.

\begin{corollary}\label{cor:chang-longer-range}
Fix $\nu,\tau>0$. There is a constant $c_3(\nu,\tau)>0$,
which depends only on $\nu$ and $\tau$, 
such that the following holds. Put
$$
\qfl\defeq
P(q)^{1000}+\exp\Bigl(\frac{c_3(\nu,\tau)\log q}{\log\log q}\Bigr)
\mand
\xi\defeq\min\{1,\tfrac{1}{3\nu}\}.
$$
For any primitive character $\chi$ modulo $q$, if
$\qfl<N< M\le 2N$ and $N\le q^\nu$, then
\be\label{eq:SCS2}
\biggr|\sum_{N<n\le M}\chi(n)n^{-\sigma-it}\biggr|
\le 4N^{1-\sigma}e^{-\xi\sqrt{\log N}}
\qquad(\tfrac12<\sigma<1,~|t|\le 3q^\tau).
\ee
\end{corollary}

\begin{proof}
By partial summation, it suffices to show that
if $\cI$ is an arbitrary interval whose
length $N$ lies in $(\qfl,q^\nu]$, then
\be\label{eq:SCS3}
\biggl|\sum_{n\in\cI}\chi(n)n^{-it}\biggr|
\le N e^{-\xi\sqrt{\log N}}\qquad(|t|\le 3q^\tau).
\ee
We apply Lemma~\ref{lem:chang-char-bound} with $T\defeq 3q^\tau$.
With this choice, the inequality
$$
(\log qT)^{1-c_1}+c_2\log\Bigl(\frac{2\log q}{\log \fq}\Bigr)
\frac{\log \fq}{\log\log q}\le\frac{c_3(\nu,\tau)\log q}{\log\log q}
$$
clearly holds if $c_3(\nu,\tau)$ is large enough. 
In the case that $\qfl<N<\tfrac12q$
we obtain~\eqref{eq:SCS1}, which clearly implies~\eqref{eq:SCS3}.
When $\tfrac12q\le N\le q^\nu$, we put
$k\defeq\fl{2N/q}$, $L\defeq N/k$, and
split $\cI$ into a sum of $k$ disjoint subintervals
of length $L$. Since $q>L\ge\frac12q$ we can use~\eqref{eq:SCS1} to bound the sum over each subinterval; this gives
$$
\biggl|\sum_{n\in\cI}\chi(n)n^{-it}\biggr|
\le kLe^{-\sqrt{\log L}}.
$$
Since $kL=N$ and
$$
\log L\ge\log(\tfrac12q)\ge\tfrac{1}{3\nu}\log q^\nu
\ge\tfrac{1}{3\nu}\log N,
$$
we obtain~\eqref{eq:SCS3} in this case as well.
\end{proof}

\subsection{$L$-function bounds and zero-free regions}

\begin{corollary}\label{cor:primLfuncbound}
Fix $\tau\ge 1$. There is a constant $c_4(\tau)>0$,
which depends only on $\tau$, such that the following holds.
Put $\nu\defeq 8\tau$, and let $c_3(\nu,\tau)>0$ have the
property described in Corollary~\ref{cor:chang-longer-range}.
For every primitive character $\chi$ of modulus $q>c_4(\tau)$ we have
$$
\big|L(s,\chi)\big|\le \eta^{-1}\qfl^\eta
\qquad(\sigma>1-\eta,~|t|\le 3q^\tau),
$$
where 
$$\eta\defeq\ell^{-1/2}(\log 2\ell)^{-3/4}, \quad 
\text{with} \ \ell\defeq\log q(|t|+3),
$$ 
and
$$
\qfl\defeq
P(q)^{1000}+\exp\Bigl(\frac{c_3(\nu,\tau)\log q}{\log\log q}\Bigr).
$$
\end{corollary}

\begin{proof}
Since $|t|\le 3q^\tau$ and $\tau\ge 1$, it is easy to check that
$$
2\ell=2\log q(|t|+3)\le 8\tau=\nu,
$$
and thus $q^\nu\ge e^{2\ell}$ holds throughout the region
$\{\sigma>1-\eta,~|t|\le 3q^\tau\}$.
As in the proof of~\cite[Lemma~8]{Iwan} we get that
$$
\biggl|\sum_{n\le \qfl}\chi(n)n^{-s}\biggr|
<\eta^{-1}(\qfl^{\eta}-1)+1\mand
\biggl|\sum_{n>Z}\chi(n)n^{-s}\biggr|
<1\qquad(Z\ge q^\nu).
$$
If $q^\nu\le \qfl$ then we are done.
Otherwise, we split the interval $(\qfl,q^\nu]$
into dyadic subintervals 
and apply Corollary~\ref{cor:chang-longer-range}
to bound the sum over each subinterval.
For $\qfl<N<M\le 2N$ and $N\le q^\nu$,
we have by~\eqref{eq:SCS2}:
$$
\biggl|\sum_{N<n\le M}\chi(n)n^{-s}\biggr|
\le b(N)\defeq 4N^\eta e^{-\xi\sqrt{\log N}}.
$$
By calculus, $b(N)$ is decreasing for
$$
N<\Omega\defeq \exp(2\eta^{-2} \xi^2) = \exp\bigl(\tfrac14\xi^2\ell(\log 2\ell)^{3/2}\bigr),
$$
hence at all points $N\in(\qfl,q^\nu]$ provided that $c_4(\tau)$
is large enough. For any such $N$ we have
$$
\log b(N)\le\log b(\qfl)=\log 4+\eta\log\qfl-\xi\sqrt{\log\qfl}
<\log 4-\tfrac12\xi\sqrt{\log\qfl},
$$
where the last inequality follows from $\qfl<\Omega$, and
we conclude that
$$
b(N)\le\frac{1}{10\nu\log q}\qquad\bigl(N\in(\qfl,q^\nu]\bigr)
$$
holds provided that
$$
\log(40\nu\log q)\le\tfrac12\xi\sqrt{\log\qfl}.
$$
Since  by definition
$$
\log\qfl>\frac{c_3(\nu,\tau)\log q}{\log\log q},
$$ 
the latter condition
is verified if $c_4(\tau)$ is large enough. Finally, summing the  
contributions from all subintervals, we find that
$$
\biggl|\sum_{\qfl<n\le q^\nu}\chi(n)n^{-s}\biggr|
\le\frac{\nu\log q}{\log 2}\cdot\frac{1}{10\nu\log q}<1,
$$
and the result follows.
\end{proof}

\begin{lemma}
\label{lem:Iwaniec1}
Let $\eta\in(0,\tfrac13)$, $M\ge e$, and put
$$
\Theta\defeq\eta^{-1}\log M,\qquad
\vartheta\defeq\frac{1}{400\Theta}.
$$
Let $q\ge 3$, and suppose that
\be\label{eq:eta condition}
8\log(5\log 3q)+24\eta^{-1}\log(160\Theta)
\le\tfrac83\Theta.
\ee
There is at most one nonprincipal character $\chi$ modulo $q$ 
such that simultaneously
\begin{itemize}
\item[$(i)$] $|L(s,\chi)|\le M$ holds for all $s=\sigma+it$ with
$\sigma>1-\eta$ and $|t|\le 3T$;
\item[$(ii)$] $L(s,\chi)$ has a zero $\beta+i\gamma$ with
$\beta>1-\vartheta$ and $|\gamma|\le T$.
\end{itemize} 
Such a zero, if it exists, is unique, simple and real.
\end{lemma}

\begin{proof}
The inequality~\eqref{eq:eta condition} is equivalent to
\be\label{eq:eta-condition2}
8\log(5\log 3q)+\frac{24}{\eta}\log(2M/5\vartheta)
\le\frac{1}{15\vartheta}.
\ee
The first part of the proof of~\cite[Lemma~11]{Iwan}
shows that $L(s,\chi)\ne 0$ throughout the region
$$
\Gamma\defeq\begin{cases}\{ \sigma+it:
\sigma>1-\vartheta,~|t|\le T\}&\quad\hbox{if $\chi^2\ne\chi_0$},\\
\{\sigma+it: \sigma>1-\vartheta,~\eta/4<|t|\le T\}&\quad\hbox{if $\chi^2=\chi_0$},\\
\end{cases}
$$ 
provided that
$$
6\log(5\log 3q)+\frac{16}{\eta}\log(M/5\vartheta)
+\frac{8}{\eta}\log(2M/5\vartheta)
\le\frac{1}{15\vartheta}.
$$  
The second part of the proof of~\cite[Lemma~11]{Iwan}
shows that $L(s,\chi)$ has at most one zero
in the region
$$
\Delta\defeq\{\sigma+it:\sigma>1-\vartheta,~|t|\le\eta/4\},
$$
and any such zero is simple and real provided that
$$
8\log(5\log 3q)+\frac{16}{\eta}\log(M/5\vartheta)
\le\frac{1}{15\vartheta}.
$$
Finally,~\cite[Lemma~12]{Iwan} asserts that there is at most
one character $\chi\ne\chi_0$ such that $L(s,\chi)$ has a real zero
$\beta>1-\vartheta$ provided that
$$
2\log(5\log 3q)+\frac{12}{\eta}\log(M/5\vartheta)
\le\frac{2}{15\vartheta}.
$$
In view of~\eqref{eq:eta-condition2} the above three
inequalities hold, and since for any  $\chi\ne\chi_0$ we have
$$
\Gamma \cup \Delta =\{ \sigma+it:
\sigma>1-\vartheta,~|t|\le T\},
$$
the result follows.
\end{proof}

Finally, we use the following statement,
which is an immediate consequence
of a result of Iwaniec~\cite[Theorem~2]{Iwan}.

\begin{lemma}\label{lem:gulp}
For any $q\ge 3$, there is no primitive character
$\chi$ modulo $q$ for which $L(s,\chi)$ has a zero
$\beta+i\gamma$ satisfying
$$
\beta>1-\frac{1}{40000(\log q+(\ell\log 2\ell)^{3/4})}
\mand \gamma\ne 0,
$$
where $\ell\defeq\log q(|\gamma|+3)$.
\end{lemma}


{\Large\section{Proofs of the main results}}

\subsection{Proof of Theorem~\ref{thm:chang-style}}

Let $\cQ$ be the set of numbers that satisfy the
conditions of Theorem~\ref{thm:chang-style} with
some large $Q_A>0$. Let $q\in\cQ$ with $q>Q_A$,
and observe that the condition~\eqref{eq:ch-restrict}
of Theorem~\ref{thm:chang-style} implies the
condition~\eqref{eq:SQA-ass} of Hypothesis~$\soundconj{\cQ,A}$.
Thus, by Proposition~\ref{prop:SQA-criterion},
to prove the desired result
it suffices to establish \eqref{eq:k0defined};
that is, we need to show that
$$
\cA\defeq
\Bigl\{\chi\ne\chi_0:L(s,\chi)=0\text{~has a zero in~}
\{\sigma>1-k_0/\log q,~|t|\le q\}\Bigr\}.
$$
has cardinality at most one.

First, we claim that there is a sufficiently large
constant $\tau_A>0$ (depending only on $A$) with the following
property. For every character $\chi$ modulo $q$, let
$\check q$ be the conductor of $\chi$, and let $\check\chi$ be
the character modulo $\check q$ that induces $\chi$.
Put
$$
\cR_\chi\defeq
\bigl\{\sigma+it:\sigma>1-k_0/\log q,
~|t|\le \min\{q,\check q^{\tau_A}\}\bigr\}.
$$
Then
\be\label{eq:A-new-description}
\cA=\{\chi\ne\chi_0:L(s,\chi)=0\text{~has a zero in~}\cR_\chi\}.
\ee
To prove the claim, suppose on the contrary that there is a
character $\chi\ne\chi_0$ such that $L(s,\chi)$ has
a zero $\beta+i\gamma$ satisfying
\be\label{eq:darksideofmoon}
\beta>1-\frac{k_0}{\log q},\qquad\min\{q,\check q^{\tau_A}\}<|\gamma|\le q.
\ee
Clearly, this is not possible unless $\check q^{\tau_A}<q$,
which we assume. Put $\ell\defeq q(|\gamma|+3)$
and $\check\ell\defeq\check q(|\gamma|+3)$, and note that
$$
40000(\log \check q+(\check\ell\log 2\check\ell)^{3/4})
\le 40000(\tau_A^{-1}\log q+(\ell\log 2\ell)^{3/4})
\le k_0^{-1}\log q
$$
if $\tau_A$ is large enough, since $\ell\ll\log q$
and $q>\check q^{\tau_A}>2^{\tau_A}$ (we remind the reader
that, in the definition~\eqref{eq:k0defined} of $k_0$, the constant 
$D$ is absolute). Therefore,~\eqref{eq:darksideofmoon} implies
$$
\beta>1-\frac{1}{40000(\log \check q
+(\check\ell\log 2\check\ell)^{3/4})}
\mand \gamma\ne 0.
$$
As this contradicts Lemma~\ref{lem:gulp} (with
$\check\chi,\check\ell,\check q$ replacing $\chi,\ell,q$
respectively),
we conclude that~\eqref{eq:A-new-description} is a correct
description of the set $\cA$.

Next, we eliminate from $\cA$ certain 
characters with a bounded conductor.
Let $c_4(\tau_A)>0$ be the constant described in
Corollary~\ref{cor:primLfuncbound} for $\tau\defeq\tau_A$. 
Let $\chi\in\cA$, and suppose that
$\check q\le c_4(\tau_A)$. Since $L(s,\chi)$ has a zero
$\beta+i\gamma\in\cR_\chi$, using the inequalities
$$
\log q \ge \log Q_A
$$
and
$$
\log\check q(|\gamma|+3)
\le\log\check q(\check q^{\tau_A}+3)
\le\log(6\check q^{\tau_A+1})
< (\tau_A+3)\log\check q
\le (\tau_A+3)\log c_4(\tau_A),
$$ 
we derive that
\dalign{
\beta &>1-\frac{k_0}{\log q}\ge 1-\frac{k_0}{\log Q_A}\\
&\ge 1-\frac{k_0}{\log Q_A}\frac{(\tau_A+3)\log c_4(\tau_A)}
{\log\check q(|\gamma|+3)}
>1-\frac{k_0(\tau_A+3)\log c_4(\tau_A)/\log Q_A}{\log\check q(|\gamma|+3)}.
} 

Since $Q_A$ can be chosen \emph{after} both $\tau_A$ and
$ c_4(\tau_A)$ 
are defined, and $\beta+i\gamma$ is  a zero of $L(s,\check\chi)$,
taking $Q_A$ large enough and applying
Lemma~\ref{lem:gron-land-titch} we deduce that $\gamma=0$.
Consequently, the real zero $\beta$ of $L(s,\chi)$ satisfies
$$
\beta>1-\frac{k_0(\tau_A+3)\log c_4(\tau_A)/\log Q_A}{\log 9}.
$$
However, this situation is untenable if $Q_A$ is sufficiently large,
for there are only finitely many characters $\chi$
modulo $q$ with a conductor $\check q\le c_4(\tau_A)$, and the
\text{$L$-function}
attached to any one of these characters has at most finitely many
zeros in the real interval 
$[0, 1]$; such zeros
must lie in $(0,1)$, hence they are bounded away from one
by a constant that depends only on $\tau$. 
In summary, if $Q_A$ is large enough, then every
$\chi\in\cA$ has $\check q> c_4(\tau_A)$,
and so~\eqref{eq:A-new-description} transforms to
\be\label{eq:A-new-description2}
\cA=\{\chi\ne\chi_0:\check q> c_4(\tau_A)\text{~and~}
L(s,\chi)=0\text{~has a zero in~}\cR_\chi\}.
\ee

It remains to show that $\cA$ defined by~\eqref{eq:A-new-description2} has cardinality at most one: 
\be\label{eq:A le 1}
|\cA| \le 1.
\ee
To this end, let $\chi\in\cA$. Applying
Corollary~\ref{cor:primLfuncbound} with $\tau\defeq\tau_A$ we have
$$
\big|L(s,\check\chi)\big|
\le \check\eta^{-1}\cqfl^{\check\eta}
\qquad(\sigma>1-\check\eta,~|t|\le 3\check q^{\tau_A}),
$$
where 
$$\check\eta\defeq\check\ell^{-1/2}(\log 2\check\ell)^{-3/4} \qquad 
\text{with}  \quad \check\ell\defeq\log \check q(|t|+3),
$$
and
\be\label{eq:funny-definition}
\cqfl\defeq
P(\check q)^{1000}+\exp\Bigl(\frac{c_3(\nu,\tau)\log \check q}
{\log\log \check q}\Bigr).
\ee
We apply Lemma~\ref{lem:Iwaniec1} with 
$$
M\defeq\check\eta^{-1}\cqfl^{\check\eta},\qquad 
\Theta\defeq\check\eta^{-1}\log M, \qquad T\defeq\check q^{\tau_A}.
$$
To establish (the analogue of) the bound~\eqref{eq:eta condition},
it suffices to show
\be\label{eq:cat-lily}
\log(5\log 3\check q)\le\tfrac16\Theta
\mand
3\check\eta^{-1}\log(160\Theta)\le\tfrac16\Theta,
\ee 
where
$$
\Theta\defeq\check\eta^{-1}\log M
=\log\cqfl-\check\eta^{-1}\log\check\eta.
$$
In $\cR_\chi$ we have $|t|\le 3\check q^{\tau_A}$, so that
\be\label{eq:supercat}
\check\ell\asymp\log\check q\mand
\check\eta\asymp(\log\check q)^{-1/2}(\log\log\check q)^{-3/4},
\ee
where the implied constants can be made explicit and
depend only on $A$. In view of~\eqref{eq:funny-definition} we deduce that
$$
\Theta
=\log\cqfl+O((\log\check q)^{1/2}(\log\log\check q)^{7/4})
=\log\cqfl+O((\log\cqfl)^{2/3}).
$$
Increasing the value of $ c_4(\tau_A)$ if necessary,
the first inequality in~\eqref{eq:cat-lily} follows from
$$
\log(5\log 3\check q)\le\tfrac1{10}\log\cqfl,
$$
which is clear if $ c_4(\tau_A)$ is large enough in view
of~\eqref{eq:funny-definition}. 
Indeed, it suffices that $c_3(\nu,\tau)\ge 1$ and that
$ c_4(\tau_A)$ exceeds a certain absolute constant, and this
can all be arranged before the value of $Q_A$ is chosen.
Next, taking into account
the second estimate of~\eqref{eq:supercat}, we see that there is
a constant $C_A>0$ (depending only on $A$) such that
the second inequality in~\eqref{eq:cat-lily} follows from
$$
(\log\check q)^{1/2}(\log\log\check q)^{7/4}\le C_A\log\cqfl
\qquad(\check q> c_4(\tau_A)),
$$
and this inequality is also clear (for large $c_4(\tau)$) by~\eqref{eq:funny-definition}.

The preceding argument shows that every character $\chi\in\cA$
satisfies~\eqref{eq:eta condition} and the condition $(i)$
of Lemma~\ref{lem:Iwaniec1}. 

We claim that the condition $(ii)$ also holds when
\be\label{eq:screech}
P(q)^{Q_A}<q
\ee
holds with some suitably large number $Q_A>0$.
To prove the claim,
suppose that \eqref{eq:screech} holds. Since $\chi\in\cA$
we see that $L(s,\chi)$ has a zero $\beta+i\gamma$ satisfying
$$
\beta>1-\frac{k_0}{\log q},\qquad
|\gamma|\le \min\{q,\check q^{\tau_A}\}.
$$
Since $T\defeq\check q^{\tau_A}$, it follows that
$$
\beta>1-\vartheta,\qquad|\gamma|\le T,
$$
and so the condition $(ii)$ is satisfied, provided that
$$
k_0\le \vartheta\log q=\frac{\log q}{400\Theta}
=(1+o(1))\frac{\log q}{400\log\cqfl}\qquad( c_4(\tau_A)\to\infty).
$$
Since $\cqfl\le\qfl$, for large enough $ c_4(\tau_A)$ it suffices to
have $q>\qfl^{500k_0}$. Recalling the definition of $\qfl$,
we see that  a value 
$$
Q_A \ge 500000k_0=500000(4A\log A+D)
$$ 
in~\eqref{eq:screech} ensures that the condition $(ii)$
of Lemma~\ref{lem:Iwaniec1} holds.

Above, we have shown that any characters in $\cA$ satisfy all of
the conditions of Lemma~\ref{lem:Iwaniec1}. By the lemma,
there is at most one nonprincipal character
meeting these conditions, thus we obtain 
\eqref{eq:A le 1}. This completes the proof of
Theorem~\ref{thm:chang-style}.

\subsection{Proof of Theorem~\ref{thm:siegel-style}}

Let $\cQ$ be the set of numbers that satisfy the
conditions of Theorem~\ref{thm:siegel-style} with
some large $Q_A>0$. Let $q\in\cQ$ with $q>Q_A$,
and observe that the condition~\eqref{eq:restrictions2}
of Theorem~\ref{thm:siegel-style} agrees with the 
condition~\eqref{eq:SQA-ass} of Hypothesis~$\soundconj{\cQ,A}$.

If $Q_A^{-1}\le c_1$, then
Lemma~\ref{lem:gron-land-titch} asserts that there is
at most one nonprincipal character $\chi$ modulo $q$ such that
$L(s,\chi)$ has a zero $\beta+i\gamma$ satisfying~\eqref{eq:except-zero-beta}, and in this case
$\beta>1-Q_A^{-1}/\log q$ and $\gamma=0$. Let $\chi_\bullet$
be such a character. Taking $\eps\defeq Q_A^{-1}$,
Lemma~\ref{lem:linnik} asserts that for any nonprincipal character
$\chi\ne\chi_\bullet$ the $L$-function $L(s,\chi)$
does not vanish in the region defined by
$$
\sigma>1-\frac{c_2\log Q_A}{\log q(|t|+3)}.
$$
In particular, if $Q_A$ is large enough, then
$L(s,\chi)$ cannot have a zero
$\beta+i\gamma$ with
$$
\beta>1-\frac{\rf{4A\log A+D}}{\log q},\qquad
|\gamma|\le q;
$$
therefore $\soundconj{\cQ,A}$ is true by
Proposition~\ref{prop:SQA-criterion}.


\begin{thebibliography}{99}

\bibitem{BanksShpar1}
W.~D.~Banks and I.~E.~Shparlinski,
Bounds on short character sums and $L$-functions 
for characters with a smooth modulus.
{\it J.\ d'Anal.\ Math.\/} 139 (2019), 239--263. 

\bibitem{BanksShpar2}
W.~D.~Banks and I.~E.~Shparlinski,
Sums with the Mobius function twisted by characters with powerful moduli. 
{\it Trans.\ Amer.\ Math.\ Soc.\/} 373 (2020), 249--272.

\bibitem{Chang}
M.-C.~Chang,
Short character sums for composite moduli.
\emph{J.\ Anal.\ Math.} 123 (2014), 1--33.

\bibitem{delaBTenen}
R.~de la Bret\`eche and G.~Tenenbaum,
Propri\'et\'es statistiques des entiers friables.
{\it Ramanujan J.\/} 9 (2005), 139--202.

\bibitem{Gal}
P.~X.~Gallagher,
Primes in progressions to prime-power modulus.
{\it Invent.\ Math.} 16 (1972), 191--201.

\bibitem{Gran1}
A.~Granville,
Integers, without large prime factors, in arithmetic progressions. I.
{\it Acta Math.\/} 170 (1993), 255--273.

\bibitem{Gran2}
A.~Granville,
Integers, without large prime factors, in arithmetic progressions. II.
{\it Philos.\ Trans.\ Roy.\ Soc.\ London Ser.\ A} 345 (1993),  349--362.

\bibitem{GranShao}
A.~Granville and X. Shao,    Bombieri-Vinogradov for multiplicative functions, and beyond the $x^{1/2}$-barrier.
{\it Adv.\  Math.} 345 (1993),    350 (2019), 304--358.

\bibitem{Gron}
T.~H.~Gronwall,
Sur les s\'eries de Dirichlet correspondant \'a
des charact\'eres complexes.
{\it Rendiconti di Palermo} 35 (1913), 145--159.

\bibitem{Harper} A.~J.~Harper,
On a paper of K.~Soundararajan on smooth numbers
in arithmetic progressions.
{\it J.\ Number Theory\/} 132 (2012), 182--199. 

\bibitem{Harper2} A.~J.~Harper, 
Bombieri-Vinogradov and Barban-Davenport-Halberstam type theorems for smooth numbers.
{\it Preprint\/}, 2012  (available
from \url{http://arxiv.org/abs/:1208.5992}).

\bibitem{HilTen1}
A.~Hildebrand and G.~Tenenbaum,
On integers free of large prime factors.
{\it Trans.\ Amer.\ Math.\ Soc.\/} 296 (1986), 265--290.

\bibitem{HilTen2}
A.~Hildebrand and G.~Tenenbaum,
Integers without large prime factors.
{\it J.\ Th\'eorie des Nombres de Bordeaux\/} 5 (1993), 411--484.

\bibitem{Iwan} H.~Iwaniec,
On zeros of Dirichlet's $L$ series.
\emph{Invent.\ Math.}   23 (1974), 97--104.

\bibitem{IwanKowal}
H.~Iwaniec and E.~Kowalski,
{\it Analytic number theory.\/}
Amer.\ Math.\ Soc., Providence, RI, 2004.

\bibitem{Kh} R.~Khan,
The divisor function in arithmetic progressions modulo prime powers.
{\it Mathemat.\/}  62 (2016), 898--908.

\bibitem{Landau}
E.~Landau,
\"Uber die Klassenzahl imagin\"ar-quadratischer Zahlk\"orper.
{\it G\"ottinger Nachrichten} (1918), 285--295.

\bibitem{Linnik}
Yu.~V.~Linnik,
On the least prime in an arithmetic progression.\ I.\ The
basic theorem.
{\it Rec.\ Math.\ (Mat.\ Sbornik) N.S.} 15(57) (1944), 347--368.

\bibitem{LSZ} K.~Liu, I.~E.~Shparlinski  and T.~P.~Zhang,
 Divisor problem in arithmetic progressions modulo a prime power.
{\it Adv.\ Math.\/}  325 (2018), 459--481.

\bibitem{Mil} D.~Mili{\'c}evi{\'c},
Sub-Weyl subconvexity for Dirichlet
$L$-functions to powerful moduli.
{\it Compos.\ Math.} 152 (2016), 825--875.

\bibitem{MilWhi} D.~Mili{\'c}evi{\'c} and D.~White,
Twelfth moment of Dirichlet $L$-functions to prime power moduli.
{\it Preprint\/}, 2019  (available
from \url{http://arxiv.org/abs/:1908.04833}).  

\bibitem{Sound}
K.~Soundararajan,
The distribution of smooth numbers in arithmetic progressions.
{\it Anatomy of Integers\/}, 115--128.
CRM Proc.\ Lect.\ Notes, vol.~46,
Amer.\ Math.\ Soc., Providence, RI, 2008.

\bibitem{Titch}
E.~C.~Titchmarsh,
A divisor problem.
{\it Rendiconti Palermo} 54 (1930), 414--429.

\end{thebibliography}
\end{document}